\documentclass[reqno,12pt]{amsart}
\usepackage{amscd,amsfonts,mathrsfs,amsthm,amssymb, amsmath}
\usepackage{enumerate}
\usepackage{bbm}
\usepackage{multicol}
\usepackage{color}
\usepackage{array}
\usepackage{ae}
\usepackage{accents}
\usepackage{epsfig}
\usepackage{pstricks, pst-node, pst-plot, pst-math}
\usepackage[e]{esvect}
\usepackage[nobysame]{amsrefs}
\psset{unit=1em}
\newcommand{\R}{\mathbb{R}}

\newcommand{\Sp}{\mathbb{S}}
\def\comment#1 {{\color{red}(Comment: #1) }}

\def\real     #1{{\mathbb R^{#1}}}
\def\complex  #1{{\mathbb C^{#1}}}
\def\natural  #1{{\mathbb N^{#1}}}

\def\lap      {\Delta }
\def\gind {\operatorname{g}}
\def\C {\mathbb{C}}
\def\W {\mathcal{W}}
\def\Sf {\mathbb{S}}
\def\dt       {\frac{d}{dt}\,}

\newtheorem{theorem}{Theorem}[section]

\newtheorem{lemma}[theorem]{Lemma}

\newtheorem*{thma}{Theorem A}

\newtheorem{proposition}[theorem]{Proposition}
\newtheorem{definition}[theorem]{Definition}
\theoremstyle{definition}
\newtheorem{remark}[theorem]{Remark}

\def\pproof#1{\@ifnextchar[\opargproof
{\opargproof[\it Proof of #1.]}}
\def\opargproof[#1]{\par\noindent {\bf #1 }}

\makeatother

\numberwithin{equation}{section}

\begin{document}

\title[MCF of Whitney spheres]{Lagrangian mean curvature flow of Whitney spheres}
\author[A. Savas-Halilaj]{\textsc{Andreas Savas-Halilaj}}
\address{%
	Andreas Savas-Halilaj\newline
	Department of Mathematics\newline
	Section of Algebra \& Geometry\newline
	University of Ioannina\newline
	45110 Ioannina, Greece\newline
	{\sl E-mail address:} {\bf ansavas@uoi.gr}
}
\author[K. Smoczyk]{\textsc{Knut Smoczyk}}
\address{%
	Knut Smoczyk\newline
	Institut f\"ur Differentialgeometrie\newline
	and Riemann Center for Geometry and Physics\newline
	Leibniz Universit\"at Hannover\newline
	Welfengarten 1\newline
	30167 Hannover, Germany\newline
	{\sl E-mail address:} {\bf smoczyk@math.uni-hannover.de}
}

\date{}
\thanks{Supported by DFG SM 78/6-1}

\begin{abstract} 
It is shown that an equivariant Lagrangian sphere with a positivity condition on its Ricci curvature develops a type-II singularity under the Lagrangian mean curvature flow that rescales to the product of a grim reaper with a flat Lagrangian subspace. In particular this result applies to the Whitney spheres.
\end{abstract}

\maketitle

\section{Introduction}
Let $\complex{m}=\R^m\times\R^m$ be the $m$-dimensional complex euclidean space and denote by $\omega$ the standard K\"ahler 2-form on $\C^m$. An immersion $F:L^m\to\C^m$ of an $m$-dimensional manifold $L^m$ is called \textit{Lagrangian}, if $F^\ast\omega=0$. It is well known that there do not exist embedded Lagrangian spheres in $\C^m$ for $m>1$. However, there are plenty of immersed Lagrangian spheres in $\C^m$. A remarkable family of such Lagrangian immersions of $\Sf^m$ in $\C^m$ is given by the \textit{Whitney spheres} (Figure \ref{fig whitney}) $\W_{R,C}:\Sf^m\subset\R\times\real{m}\to\C^m$ defined by the maps
	\begin{figure}[ht]
	\includegraphics[scale=.42]{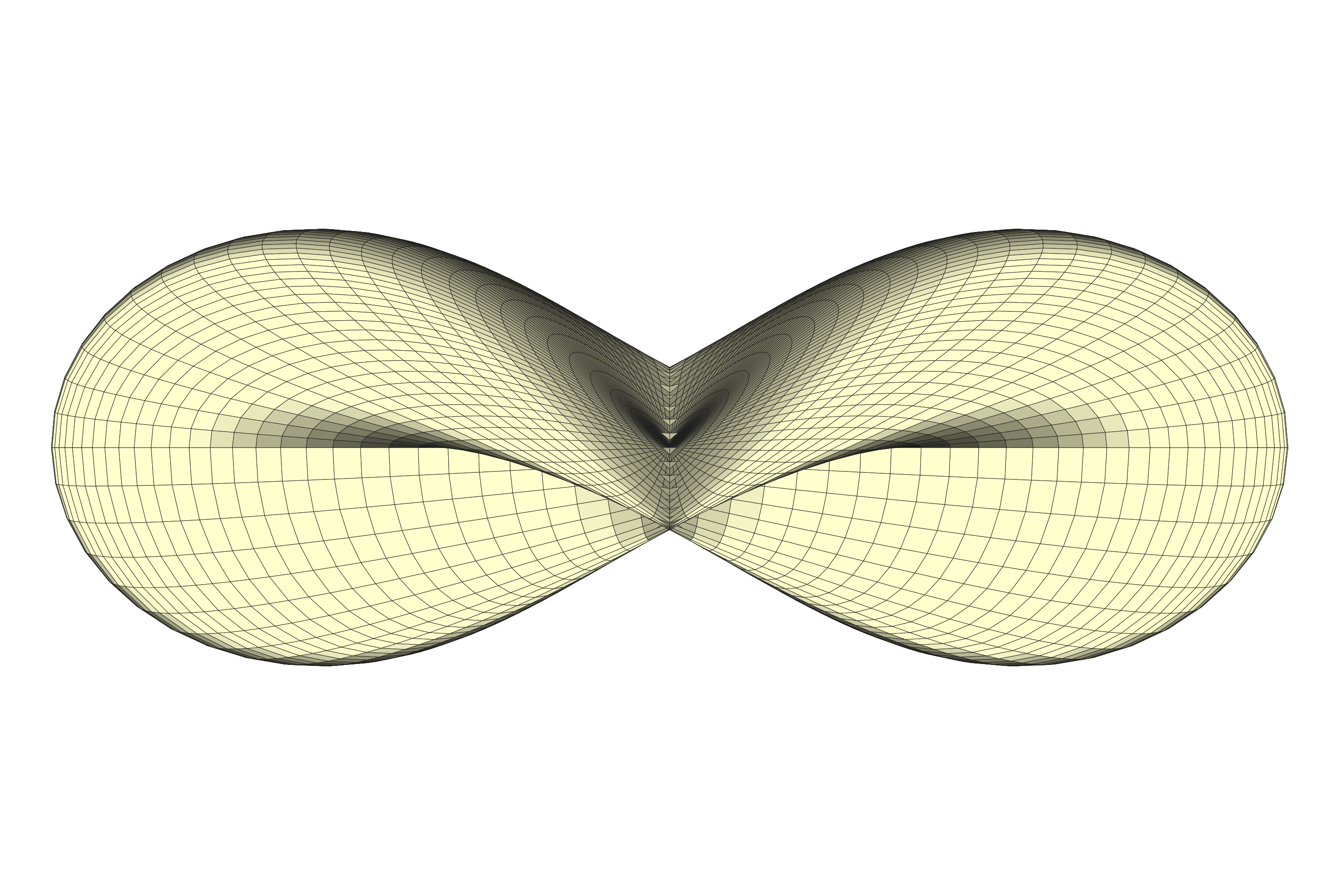}\caption{A projection of the Whitney 2-sphere.}\label{fig whitney}
	\end{figure}	
	$$\W_{R,C}(x_0,x_1,\dots,x_m)=\frac{R}{1+x_0^2}(x_1,\dots,x_m;x_0x_1,\dots,x_0x_m)+C,$$ 
	where here $R$ is a positive constant and $C$ is a point in $\C^m$. We will refer to the quantities $R$ and $C$ as the \textit{radius} and the \textit{center} of the Whitney sphere. Therefore, up to dilations and translations, all Whitney spheres can be identified with $\W=\W_{1,0}$. These spheres admit a single double point at the origin of $\C^m$. 
	
	From the definition of $\W$ one immediately observes that the Whitney sphere is invariant under the full group of isometries in $\mathbb{O}(m)$, where the action of ${\rm A}\in\mathbb{O}(m)$ on $\C^m$ is defined by ${\rm A}(x,y):=({\rm A}x,{\rm A}y)$ for all $x,y\in\R ^m$.  Whitney spheres are therefore in a sense the most symmetric Lagrangian spheres that can be immersed into $\C^m$. There do not exist any compact Lagrangian submanifolds in euclidean space with strictly positive Ricci curvature. This follows from the theorem of Bonnet-Myers, the Gau\ss\ equation for the curvature tensor and from the fact that the first Maslov class $m_1$ of the Lagrangian submanifold satisfies $m_1=[\Theta_H/\pi],$ where $\Theta_H$ denotes its mean curvature one-form. The Whitney spheres have strictly positive Ricci curvature except at the intersection point, where it vanishes. More precisely, the Ricci curvature of a Whitney sphere can be diagonalized in the form
\begin{equation}\label{eq ricwhitney}
\operatorname{Ric}=\operatorname{diag}\big({2(m-1)}r^2, mr^2,\dots,mr^2\big),
\end{equation}
	where $r$ is the distance of the point to the origin. Lagrangian submanifolds that are invariant under the full action of $\mathbb{O}(m)$ are often called \textit{equivariant} and they have been studied by many authors; see for example \cites{anciaux,chen,groh,gssz,viana}.
	
	An equivariant Lagrangian submanifold can be described in terms of its corresponding \textit{profile curve} that is defined as the intersection of the Lagrangian submanifold with the plane 
	$$\big\{(\alpha_1,0,\dots, 0;\alpha_2,0,\dots,0)\in\C^m:\alpha_1,\alpha_2\in\R\big\}.$$ 
	This profile curve is point symmetric with respect to the origin. For the Whitney sphere one obtains the \textit{figure eight curve} $z:\Sf^1\to\C$ (see Figure \ref{fig figure eight}) given by
	\begin{figure}[ht]
	\includegraphics[scale=.38]{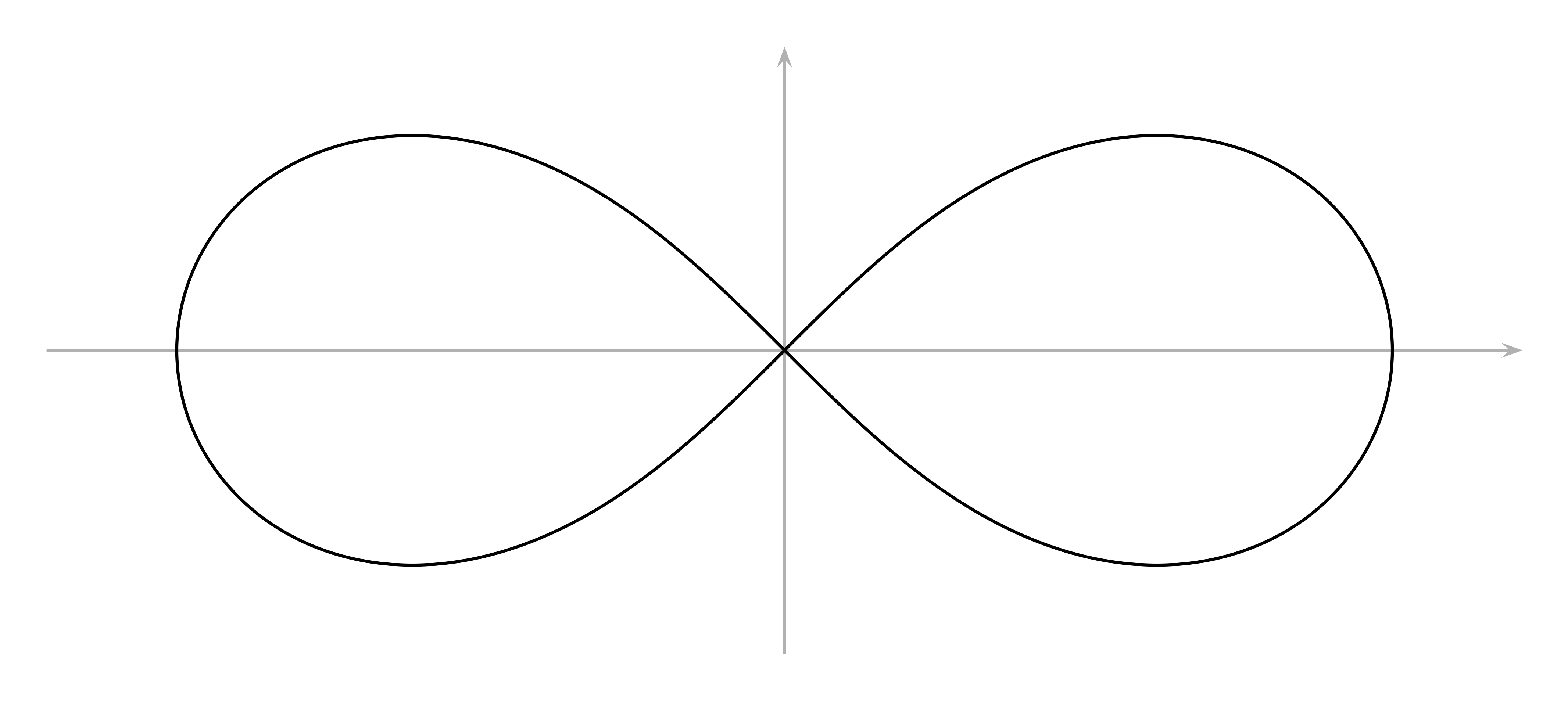}\caption{The figure eight curve is the profile curve for the Whitney sphere.}\label{fig figure eight}
	\end{figure}	
	$$z(s):=\frac{1}{1+\cos^2(s)}\bigl(-\sin(s),\sin(s)\cos(s)\bigr).$$
	The behavior of the Whitney sphere under the mean curvature flow differs significantly from that of a round sphere in euclidean space. In the matter of fact, the submanifolds that evolve from the Whitney sphere cannot be Whitney spheres anymore; they do not simply scale by a time dependent dilation, like round spheres in the euclidean space.

	This can be easily seen, because such a behavior would imply that the Whitney spheres are self-shrinking Lagrangian submanifolds with a trivial Maslov class. But as has been shown by Smoczyk \cite[Theorem 2.3.5]{smoczyk1} such self-shrinkers do not exist. More generally, it was later shown by Neves \cite{neves1} that Lagrangian submanifolds with trivial first Maslov class, so especially Lagrangian spheres of dimension $m>1$, will never develop singularities of {type-I}. It seems that type-II singularities are generic singularities in the Lagrangian mean curvature flow, even in the non-zero Maslov class case. For some recent results we refer to \cite{evans,smoczyk3}. We would like to mention that not very much is known about the classification of type-II singularities of the mean curvature flow, particularly in higher codimensions. Possible candidates are translating solitons, special Lagrangians, products of these types and various other types of eternal solutions. There exist a number of classification results for translating solitons; see for example \cite{joyce,martin,tian,xin}.

	It is therefore in particular clear that a Whitney sphere must develop a type-II singularity. The main objective of this paper is to classify its singularity in more detail. We will prove the following result.
	\begin{thma}\label{thma}
		Let $F_0:\Sf^m\to\C^m$, $m>1$, be a smooth equivariant Lagrangian immersion such that the Ricci curvature satisfies
		\begin{equation}\label{eq ricci}
		\operatorname{Ric}\ge c r^2\hspace{-3pt}\cdot\hspace{-2pt}\operatorname{g},
		\end{equation}
		where $c>0$ is a constant, $r$ denotes the distance to the origin and $\operatorname{g}$ is the induced metric. Then the Lagrangian mean curvature flow 
		$$F:\Sf^m\times[0,T)\to\C^m,\qquad\frac{dF}{dt}=H,\qquad F(\cdot ,0)=F_0(\cdot)$$
		exists for a maximal finite time $T>0$. The Lagrangian submanifolds stay equivariant and satisfy an inequality $\operatorname{Ric}\ge \varepsilon r^2\cdot\operatorname{g}$, where $\varepsilon>0$ is a time independent constant. The flow develops a type-II singularity and its blow-up converges to a product of a grim reaper (see Figure $\ref{fig grim}$) with a flat $(m-1)$-dimensional Lagrangian subspace in $\C^{m-1}$. 
	\end{thma}

	\begin{figure}[ht]
	\includegraphics[scale=.4]{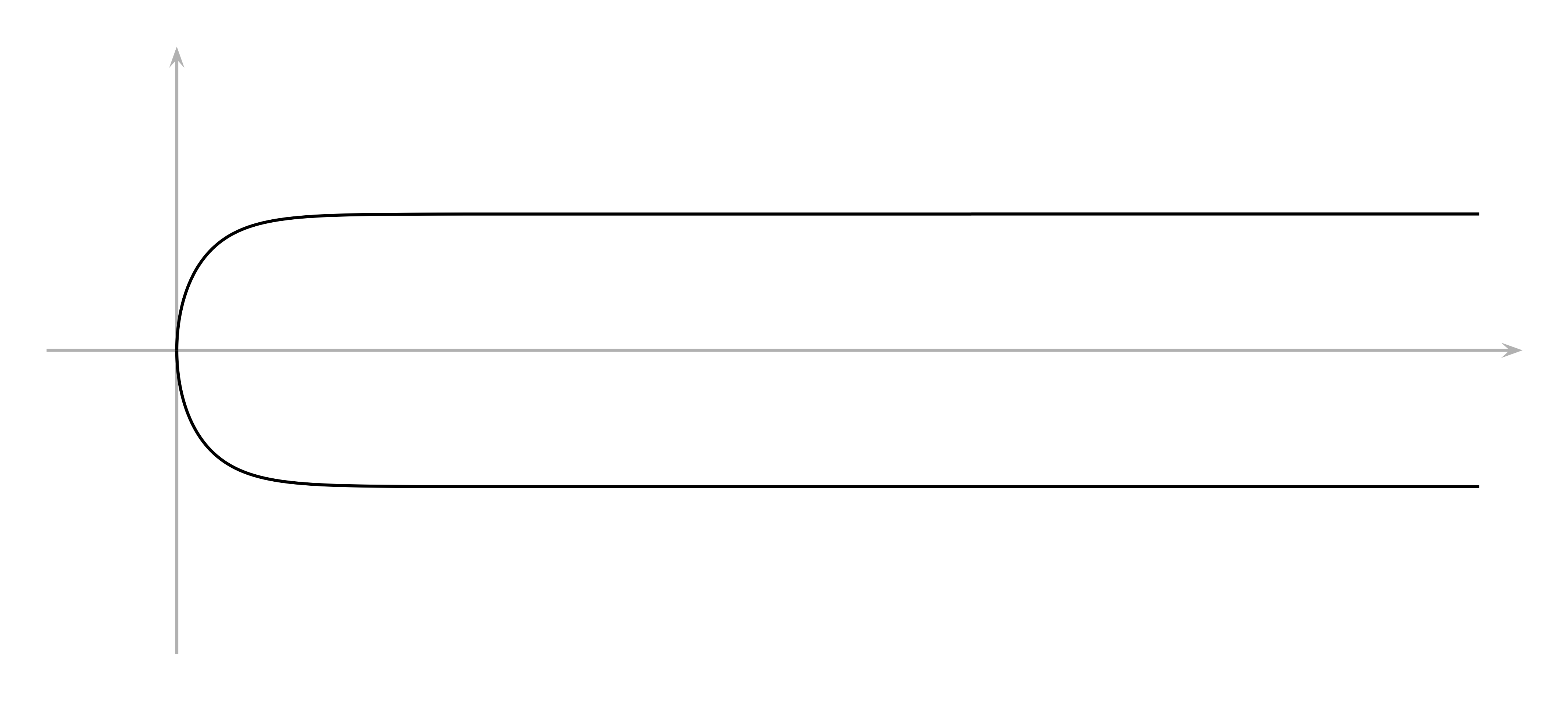}\caption{The grim reaper $\{(x,y)\in\C:\cos y=e^{-x}\}$ is a translating soliton of the mean curvature flow.
	It moves with speed $1$ in direction of the positive $x$-axis.}\label{fig grim}
	\end{figure}
	From equation \eqref{eq ricwhitney} it follows that the Whitney spheres satisfy \eqref{eq ricci} with $c=m$ and therefore our theorem includes this important case. Since the mean curvature flow is isotropic, it follows that the condition of being equivariant is preserved under the flow. That there exists a finite maximal time interval on which a smooth solution exists is also well known and by the aforementioned result by Neves it is also known that the singularity must be of type-II. The difficulty here is to understand the singularity in more detail. To achieve this, we will need the condition on the Ricci curvature which in view of equation \eqref{eq ricwhitney} is very natural. To our knowledge this is the first result on the Lagrangian mean curvature flow where a convexity assumption can be used to classify its singularities. 
	
	In dimension $m=1$ the equivariance is meaningless, since the flow reduces to the usual curve shortening flow. In this case the behavior of the figure eight curve was studied among others by Altschuler \cite{altschuler} and Angenent \cite{angenent}. In particular it was shown that the figure eight curve develops a type-II singularity that after blow-up converges to a grim reaper. 
	
	Very recently Viana \cite{viana} studied the equivariant curve shortening flow for a class $\mathcal{C}$ of antipodal invariant figure eight curves with only one self-intersection which is transversal and located at the origin. In particular he proved that the tangent flow at the origin is a line of multiplicity two, if the curve $\gamma\subset\mathcal{C}$ satisfies at least one the following two conditions: i) $\gamma\cap \mathbb{S}^1(R)$ has at most 4 points for every $R>0$; ii) $\gamma$ is contained in a cone of opening angle $\pi/2$. We would like to mention that neither the class of curves in Viana's paper nor our class of curves contain each other. For example we allow profile curves with more self-intersections or profile curves that are not contained in a cone of opening angle $\pi/2$. The essential condition in our paper is the positivity of the Ricci curvature away from the origin. Moreover the methods in both papers are substantially different from each other. 
\section{Equivariant Lagrangian Submanifolds}
\subsection{Basic facts}\label{BF}
	Consider a regular curve $z=(u,v):(-\delta,\delta)\to\C$, $\delta>0$, denote by $S$ the standard embedding of $\mathbb{S}^{m-1}$ into $\real{m}$ and define the map $F:M^m:=(-\delta,\delta)\times\mathbb{S}^{m-1}\to\C^m$ by
	$$F(s,x):=\big(u(s)S(x),v(s)S(x)\big).$$
	The curve $z$ is called the {\it profile curve} of $F$. Observe that $F$ gives rise to an immersion if its profile curve does not pass through the origin. Let us denote by $r$ the distance function with respect to the origin of $\C^m$. One can easily check that, if $\{e_1,\dots,e_{m-1}\}$ is a local orthonormal frame field of the sphere, then the vector fields
	$$
	E_0=|z'|^{-1}\partial_s,\,\, E_1=r^{-1} e_1,\dots,E_{m-1}=r^{-1}e_{m-1},
	$$
	form a local orthonormal frame on $M^m$ with respect to the Riemannian metric $\gind$ induced by $F$. Moreover, the vector fields
	$$
	\nu=-J\cdot DF(E_0),\,\,\nu_1=J\cdot DF(E_1),\dots,\,\,\nu_{m-1}=J\cdot DF(E_{m-1}),
	$$
	where $J$ is the complex structure of $\C^m$, form a local orthonormal frame field of the normal bundle of $F$. Hence $F$ is a Lagrangian submanifold. The shape operators  with respect to the above mentioned orthonormal frames have the form
	$$
	A^{\nu}=-\left(
	\begin{matrix}
	k & 0 & \cdots & 0 \\
	0 & p & \cdots & 0 \\
	\vdots & \vdots & \ddots & \vdots \\
	0 & 0 & \cdots & p\\
	\end{matrix}
	\right)
	\quad \text{and}\quad
	A^{\nu_i}=-\left(
	\begin{matrix}
	0 & p & \cdots & p \\
	p & 0 & \cdots & 0 \\
	\vdots & \vdots & \ddots & \vdots \\
	p & 0 & \cdots & 0\\
	\end{matrix}
	\right),
	$$
	for any $i\in\{1,\dots,m\}$, where $k$ denotes the curvature of the curve $z$ and $p=\langle z,\nu\rangle r^{-2}$.
	
	One can readily check that the mean curvature vector $H$ and the squared norm $|A|^2$ of the second fundamental form of $F$ are given by
	\begin{equation}\label{eq sec}
	H=-\big(k+(m-1)p\big)\nu\quad\text{and}\quad |A|^2=k^2+3(m-1)p^2,
	\end{equation}
	
	respectively. In particular, one obtains that
	\begin{equation}\label{eq pinch}
	|A|^2-\frac{3}{m+2}|H|^2=\frac{m-1}{m+2}(k-3p)^2.
	\end{equation}
	Moreover, from the Gau{\ss} equation follows:
\begin{lemma}\label{lemm ricc}
	The eigenvalues of the Ricci curvature of an equivariant Lagrangian submanifold are given by
	$$(m-1)p(k-p)\quad\text{and}\quad p(k-p+(m-2)p),$$ 
	where the latter occurs with multiplicity $m-1$.
\end{lemma}
\begin{remark}
	As can be easily computed, for the figure eight curve (see Figure \ref{fig figure eight}) we have
	$$k=3r=3p.$$
	This fact and Lemma \ref{lemm ricc} imply that on the Whitney sphere the eigenvalues of the Ricci curvature are
	$$2(m-1)r^2\quad\text {and}\quad mr^2,$$
	where the latter has multiplicity $m-1$. Observe also (compare with \eqref{eq pinch}) that for the Whitney spheres
	$$|A|^2-\frac{3}{m+2}|H|^2=0.$$
	It is well known that for any Lagrangian submanifold one always has the pinching inequality
	$$|A|^2-\frac{3}{m+2}|H|^2\ge 0.$$
	Ros and Urbano \cite[Corollary 3]{ru} showed that equality holds on a compact Lagrangian submanifold $L\subset\C^{m}$, if and only if $L$ coincides with a Whitney sphere; see also \cite{bcm} for another proof.
\end{remark}

\subsection{Equivariant Lagrangian spheres} 
	Now let $z:(-\delta,\delta)\to\C$ be a real analytic regular curve which is passing through the origin only once and such that $z(0)=0$. Then for $m>1$ one can easily verify that the tangent space of $F$, defined as above, smoothly extends over the set $\{0\}\times\Sp^{m-1}$ and gives rise to a smooth equivariant Lagrangian submanifold if and only if the curve $z$ is point symmetric. Note that in this case, the curvature $k$ of $z$ at the origin must be necessarily zero.
	
	On the other hand, if $F:\Sp^m\to\C^m$, $m>1$, is a real analytic equivariant Lagrangian immersion, then it is not difficult to see that its profile curve intersects itself at the origin only once. Since immersions are locally embeddings, each arc of the profile curve passing through the origin must be point-symmetric; because the profile curve is point-symmetric itself, it can only admit two such arcs. This does not hold for $m=1$ since the orbits in this case are not connected.
	
	We summarize this in the following lemma.
\begin{lemma}\label{lemma profile}
	If $m>1$ and $F:\Sp^m\to\C^m$ is a real analytic equivariant Lagrangian immersion, then the profile curve $\gamma$ can be parametrized by a point-symmetric real analytic regular curve $z:[-\pi,\pi]\to\C$ such that $z^{-1}(0)=\{-\pi,0,\pi\}$ and $z(-s)=-z(s)$, for all $s\in [-\pi,\pi]$.
\end{lemma}
\begin{remark}\label{rem butterfly}
	Let us make some comments.
	\begin{enumerate}[(a)]
	\item 
	The {\it butterfly curve} (Figure \ref{fig butterfly}) is an analytic point symmetric curve but not each arc passing through the origin is point-symmetric. Therefore it
	\begin{figure}[ht]
		\includegraphics[scale=.5]{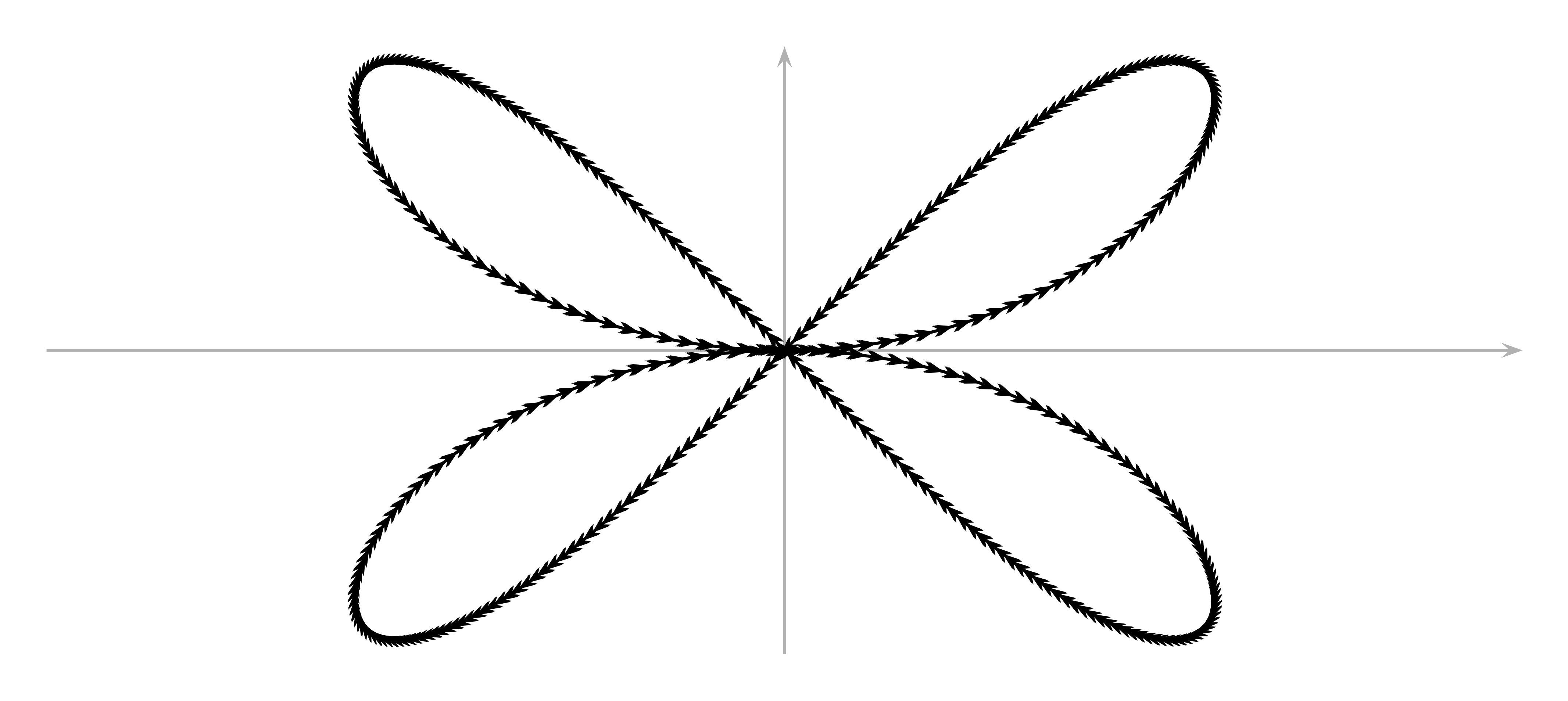}\caption{The butterfly curve.}\label{fig butterfly}
	\end{figure}
	cannot be the profile curve of an equivariant Lagrangian sphere, if $m>1$.\\	
	\item
	In view of the last lemma it is sufficient to analyze the behavior of the arc $\gamma_\ell:=z([0,\pi])$. Additionally, after rotating the curve we can assume without loss of generality that $z(s)=(x(s),y(s))$ satisfies $x(s)<0$ and $ y(s)>0$ for sufficiently small $s>0$. Throughout the paper we will always make this assumption. The arc $\gamma_\ell$ then starts and ends at the origin and is nonzero elsewhere. It does not have to be embedded and it might wind around the origin, e.g. the arc in Figure \ref{fig arc} generates a profile curve of an immersed Lagrangian sphere.
\begin{figure}[ht]
	\includegraphics[scale=.6]{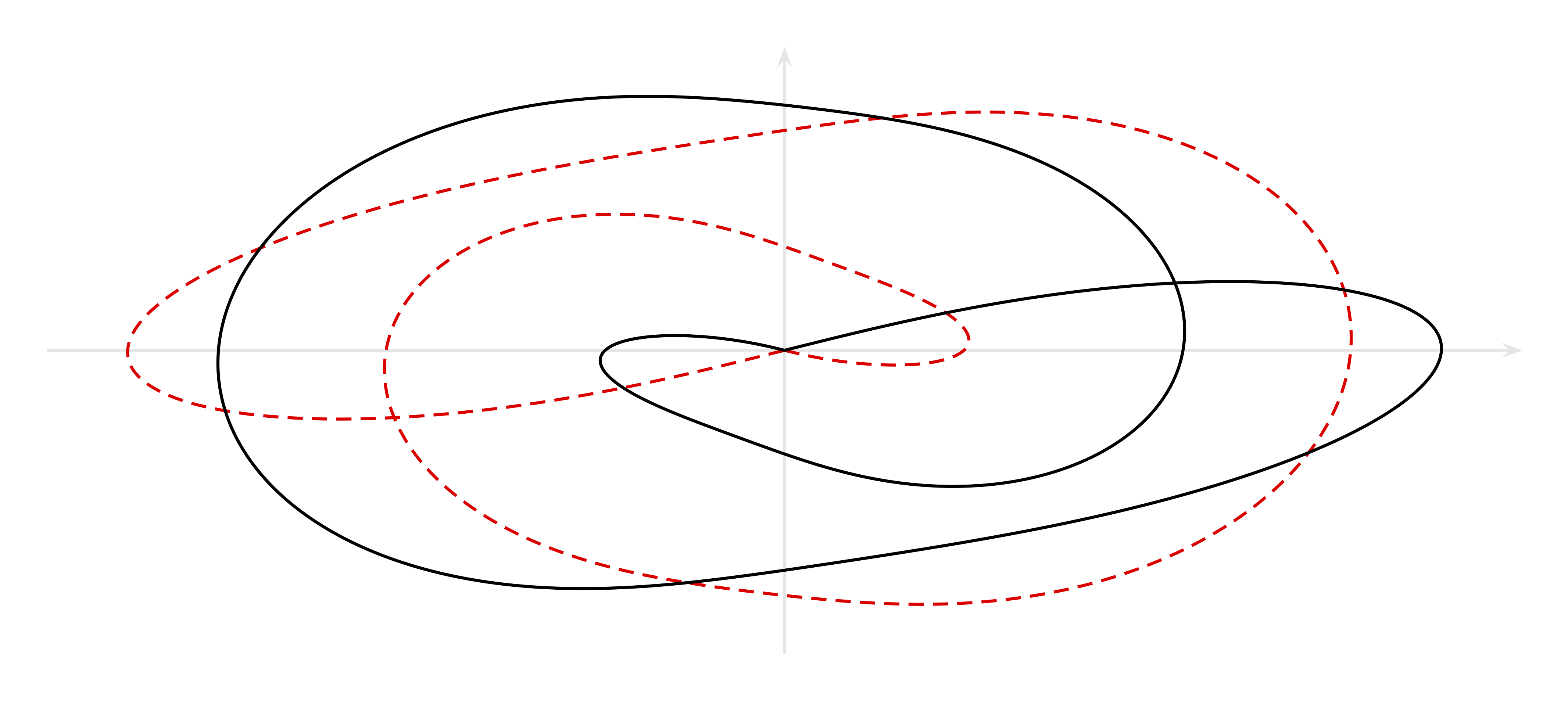}\caption{An immersed arc $\gamma_\ell$, generating a profile curve of an immersed Lagrangian sphere.}\label{fig arc}
\end{figure}
	\end{enumerate}
\end{remark}

\subsection{The signed distance function}
	Since we are interested in equivariant Lagrangian spheres, we may for the rest of the paper assume that $z:\Sf^1\to\C$ is a real analytic profile curve satisfying the conditions in Lemma \ref{lemma profile}, i.e. we identify $\Sf^1$ with the interval $[-\pi,\pi]$ and assume that $z:[-\pi,\pi]\to\C$ is a closed real analytic curve that satisfies $z(-s)=-z(s),$ for all $s\in[-\pi,\pi]$ and $z^{-1}(0)=\{-\pi,0,\pi\}$. For such curves we may introduce the signed distance function
	$$r(s):=\operatorname{sign} (s)\cdot|z(s)|.$$
\begin{lemma}\label{lemm crucial1}
	Let $r$ be the signed distance function of a real analytic curve $z:[-\pi,\pi]\to\C$ as above. Then $r$ is real analytic and satisfies 
	$$
	|\nabla r|^2=1-r^2p^2\quad\text{and}\quad\Delta r=rp(p-k).
	$$
\end{lemma}
\begin{proof}
	At points where the curve $z$ is not passing through the origin, the analyticity of $r$ is clear. To see that the distance function $r$ is real analytic at points $s_0\in\{-\pi,0,\pi\}$ we represent the curve $z$ locally around $s_0$ as the graph over its tangent line at $s_0$, i.e. locally $z$ can be represented as the graph of an analytic function $y:(-\delta,\delta)\to\real{}$ with $y(0)=y'(0)=0$. This implies $y(s)=s^2\vartheta (s)$ with another real analytic function $\vartheta$. Then the distance function $r$ is locally given by 
	$$r(s)=\operatorname{sign}(s)\sqrt{s^2+s^4\vartheta^2(s)}=s\sqrt{1+s^2\vartheta^2(s)}$$
	which clearly is real analytic. Using the frames introduced in Section \ref{BF} we see 
	$$
	2r\nabla r=\nabla r^2= \nabla|z|^2=2\langle \nabla_{E_0}z,z\rangle E_0=2|z'|^{-1}\langle z,z'\rangle E_0.
	$$
	From the decomposition
	$$
	\frac{z}{r}=\frac{\langle z,z'\rangle}{r|z'|}\frac{z'}{|z'|}+\frac{\langle z,\nu\rangle}{r}\nu
	$$
	we deduce 
	$$
	|\nabla r|^2=1-r^2p^2.
	$$
	Since $\Delta z=-k\nu$ we obtain
	$$2|\nabla r|^2+2r\Delta r=\Delta r^2=2\langle z,\Delta z\rangle+2|\nabla z|^2=-2kr^2p+2$$
	which implies the second equation.
\end{proof}
Note that the analyticity of $r$ and
$|\nabla r|^2=1-r^2p^2$
already imply that $r^2p^2$ must be analytic. From \cite[Lemma 3.3]{gssz} we also know that $\lim_{r\to 0}p=0$ since the curvature $k$ at the intersection point must vanish. We will now improve these results further, in fact we will show that $p$ itself is a real analytic function defined on the whole curve.
\begin{lemma}\label{lemm crucial}
	Let $\gamma$ be a point-symmetric real analytic arc that passes through the origin only once. 
	\begin{enumerate}[\rm (a)]
		\item 
		There exist a constant $r_0>0$ and a local point symmetric real analytic parametrization $z$ of $\gamma$ in the interval $(-r_0,r_0)$ such that for all $r\in(-r_0,r_0)$ 
		$$
		|z(r)|=|r|\quad\text{and}\quad z(r)=\sum_{n=0}^\infty z^{(2n+1)}(0)\frac{r^{2n+1}}{(2n+1)!}.
		$$
		In particular $r$ is the signed distance function of $z$ and $|z'(0)]=1$.\\
		\item
		The functions $pr^{-1}$ and $(k-3p)r^{-3}$ are real analytic on $\gamma$. In particular
		\begin{equation}\label{eq lim s}
		\lim_{r\to 0}p=0,\quad \lim_{r\to 0}\frac{k-3p}{r^2}=0.
		\end{equation}
		\item
		If $\gamma$ is not flat, then the function $k/p$ is well defined and real analytic in a neighborhood of $r=0$ and
		$$\lim_{r\to 0}\frac{k}{p}=2\ell+3,$$
		where $2\ell+1$ is the order of the zero of $p$ at $r=0$.
	\end{enumerate}
\end{lemma}
\begin{proof}
	\begin{enumerate}[(a)]
		\item 
		Observe that $|\nabla r|^2=1-r^2p^2$ and $\lim_{r\to 0}p=0$ imply that at the origin the gradient of the distance function $r$ is non-zero. Hence from the inverse function theorem, the real analyticity of $\gamma$ and the point-symmetry of the arc we conclude that $\gamma$ can locally be parameterized by a Taylor expansion 
		$$
		z(r)=\sum_{n=0}^\infty z^{(2n+1)}(0)\frac{r^{2n+1}}{(2n+1)!}
		$$
		such that $|z(r)|=|r|$. From 
		$$\left|\frac{z(r)}{r}\right|=1\quad\text { and }\quad\lim_{r\to 0}\frac{z(r)}{r}=z'(0)$$
		we get $|z'(0)|=1$.\\
		\item
		Applying part (a), the curve $z$ can be represented locally around the origin as a Taylor series in terms of the signed distance. Since $z$ is an odd function in $r$, its derivative $z'$ must be even. The same holds for the unit normal $\nu$. Hence there exists a Taylor series for $\nu$ which has the form
		$$\nu(r)=\sum_{n=0}^\infty a_{2n}r^{2n},$$
		where $a_{2n}$ are complex numbers. Since
		$$z'(0)=\lim_{r\to 0}\frac{z}{r}$$
		is the unit tangent vector to $z$ at the origin, we must have  
		$$\langle z'(0),a_0\rangle=\langle z'(0),\nu(0)\rangle=0.$$ 
		Therefore 
		$$\langle z(r),\nu(r)\rangle=\sum_{n=0}^\infty b_{2n+1}r^{2n+1},$$
		where 
		$$b_1=\langle z'(0), a_0\rangle=0.$$
		So $\langle z,\nu\rangle$ has a zero at $r=0$ of order at least $3$. This proves that $pr^{-1}=\langle z,\nu\rangle r^{-3}$ is real analytic on $\gamma$.\\
		
		\noindent
		For the second claim we first observe that
		$$r\nabla p=(k-2p)\nabla r.$$
		Representing everything in terms of functions in $r$, we deduce
		$$k-3p=rp'-p=r^2\left(\frac{p}{r}\right)'.$$
		But since $pr^{-1}$ is even, its derivative must have a zero at $r=0$ and hence the function $k-3p$ must have a zero at $r=0$ of order at least $3$. This proves the second claim.\\
		\item
		Since $p$ is real analytic we can represent $p$ locally around $r=0$ as a Taylor series in $r$. Let $2\ell+1$, $\ell\in\mathbb{N}_0$, be the order of the zero of $p$ at $r=0$. Then locally
		$$p(r)=\sum_{n=\ell}^\infty c_{2n+1}r^{2n+1}$$
		with coefficients $c_{2n+1}$. Because $k=rp'+2p$, we see that the zero of $k$ at $r=0$ has the same order $2\ell+1$. If $\gamma$ is not flat, then $p(r)\neq 0$ for sufficiently small $r\neq 0$. Therefore the quotient $k/p$ is well defined and analytic for such $r$ and the isolated singularity at $r=0$ is removable. From $k=rp'+2p$ we then immediately obtain 
		$$\lim_{r\to 0}\frac{k}{p}=\lim_{r\to 0}\frac{rp'}{p}+2=2\ell+3.$$
	\end{enumerate}
		This completes the proof.
\end{proof}
The next lemma will be important in our proofs.
\begin{lemma}\label{est1}
	For an equivariant Lagrangian sphere in $\C^m$, $m>1$,  the following conditions are equivalent.
	\begin{enumerate}[\rm(a)]
		\item 
		The Ricci curvature satisfies $\operatorname{Ric}\ge \varepsilon r^2\cdot\operatorname g$ for some positive constant $\varepsilon>0$.
		\item 
		There exist constants $\varepsilon_1>0$ and $\varepsilon_2>0$ such that the estimates
		$p\ge\varepsilon_1\,r$ and $k-p\ge \varepsilon_2\,r$ hold on the profile curve.
	\end{enumerate}
\end{lemma}
\begin{proof}
	Recall from Lemma \ref{lemm ricc} that the eigenvalues of the Ricci tensor are given by
	$$(m-1)p(k-p)\quad\text{and}\quad p(k-p)+(m-2)p^2,$$
	where the second eigenvalue has multiplicity $m-1$. Therefore condition (b) clearly implies (a). To show the converse, observe
	at first that $\operatorname{Ric}\ge \varepsilon\,r^2\cdot \operatorname{g}$ implies
	$$\frac{p}{r}\cdot\frac{k-p}{r}\ge \frac{\varepsilon}{m-1}.$$
	Hence the functions $r^{-1}p$ and $(k-p)r^{-1}$ are nonzero and have the same sign. Taking into account Lemma \ref{lemm crucial}(b), we deduce
	\begin{eqnarray*}
		\frac{\varepsilon}{m-1}&\le&\lim_{r\to 0}\left(\frac{p}{r}\cdot\frac{k-p}{r}\right)\\
		&=&\lim_{r\to 0}\left(\frac{p}{r}\cdot\frac{k-3p}{r}\right)+2\lim_{r\to 0}\frac{p^2}{r^2}\\
		&=&2\lim_{r\to 0}\frac{p^2}{r^2}.
	\end{eqnarray*}
	Therefore the function $r^{-1}p$ is bounded and non-zero everywhere. Since the profile curve is closed, there exists always a point where $r^{-1}p$ is positive. Hence there exist positive constants $\varepsilon_1,\varepsilon_2$ with $r^{-1}p\ge \varepsilon_1$ and $(k-p)r^{-1}\ge \varepsilon_2$. This completes the proof.
\end{proof}
\section{Evolution equations}
As was shown in \cite{gssz}, the mean curvature flow of equivariant Lagrangian submanifolds is fully determined by the evolution of their profile curves $z_t:\Sf^1\to\C$, $t\in[0,T)$, under the \textit{equivariant curve shortening flow} (ECSF) given by
	\begin{equation}\tag{ECSF}\label{equiv}
	\frac{dz}{dt}=-(k+(m-1)p)\nu,
	\end{equation}
where $k(s,t)$ is the curvature, $\nu(s,t)$ the \textit{outward}\footnote{Here \textit{outward} means that $\{\nu,z'\}$ forms a positively oriented  basis of $\C$.} directed unit normal of the curve $z_t$ at the point $z_t(s):=z(s,t)$ and $T>0$ is the maximal time of existence  of the solution. 

Let us define the function 
$$h:=k+(m-1)p.$$
The statements of the next Lemma were already shown in \cite[Lemma 2.2, Lemma 2.3]{smoczyk2} and \cite[Lemma 3.12]{gssz} for curves not passing through the origin. In the sequel the computations of various evolution equations are considered at points where $r\neq 0$, even though some of the quantities might extend smoothly to the origin. 
\begin{lemma}\label{lemma 4}
	Under the equivariant curve shortening flow \eqref{equiv} the following evolution equations for $z$, the normal $\nu$, the curvature $k$, the induced length element $d\mu$, the distance function $r$ and the driving term $h$ hold.
	$$\dt d\mu=-hk \, d\mu,\quad\frac{d\nu}{dt}=\nabla h,\quad\frac{d k}{dt}=\lap h+hk^2,$$
	\smallskip
	$$\frac{d z}{dt}=\Delta z-(m-1)p\nu,\quad \frac{dr}{dt}=\lap r-mrp^2$$
	and
	\begin{eqnarray*}
	\frac{dh}{dt}&=&\lap h+(m-1)\langle\nabla h,\nabla \log r\rangle\nonumber\\
	&&+h\left\{-(m-1)|\nabla \log r|^2+\frac{1}{m}h^2+\frac{m-1}{m}(k-p)^2\right\}.
	\end{eqnarray*}
\end{lemma}
Additionally, for any real number $\alpha$ we set
$$p_\alpha:=r^{-1-\alpha}p,\quad h_\alpha:=r^{-1-\alpha}h,\quad q_\alpha:=h_\alpha-mp_\alpha=(k-p)r^{-1-\alpha}.$$
\begin{lemma}\label{lemma 4b}
The functions $p_\alpha$, $h_\alpha$ and $q_\alpha$ satisfy the following evolution equations.
	\begin{eqnarray*}
	\frac{dp_\alpha}{dt}&=&\Delta p_\alpha+(2\alpha+m+1)\langle\nabla p_\alpha,\nabla \log r\rangle
	+2(q_\alpha-2p_\alpha){|\nabla \log r|^2}\nonumber\\
	&&+p_\alpha\Big\{\alpha(\alpha+m)|\nabla \log r|^2+m(\alpha+2)p^2+(k-p)^2\Big\},\\
	\frac{dh_\alpha}{dt}&=&\Delta h_\alpha+(2\alpha+m+1)\langle\nabla h_\alpha,\nabla \log r\rangle+2phq_\alpha\\
	&&+h_\alpha\Big\{\alpha(\alpha+m)|\nabla \log r|^2+m(\alpha+2)p^2+ (k-p)^2\Big\},\\
	\frac{dq_\alpha}{dt}&=&\Delta q_\alpha+(2\alpha+m+1)\langle\nabla q_\alpha,\nabla \log r\rangle\nonumber\\
	&&+\Big\{\big(\alpha(\alpha+m)-2m\big)q_\alpha+4mp_\alpha\Big\}|\nabla \log r|^2\\
	&&+q_\alpha\Big\{m(\alpha+2)p^2+ (k-p)^2+2ph\Big\}.
	\end{eqnarray*}
\end{lemma}
\begin{proof}
For $p_\alpha$ we can use the evolution equation in the proof of \cite[Lemma 3.26]{gssz}. This gives in a first step 
	\begin{eqnarray}\label{qa}
	&&\frac{dp_\alpha}{dt}=\Delta p_\alpha+(m-1)\langle\nabla p_\alpha,\nabla \log r\rangle
	+2(\alpha+2)q_\alpha|\nabla \log r|^2\\
	&&+p_\alpha\left\{-\Bigl((\alpha+2)(\alpha+4)+2m\Bigr)|\nabla \log r|^2+\frac{m(\alpha+2)}{r^2}+(k-p)^2 \right\}.\nonumber
	\end{eqnarray}
	From the identity
	\begin{equation}\label{eq gradpa}
	\nabla p_\alpha=(q_\alpha-(\alpha+2)p_\alpha)\nabla \log r
	\end{equation}
	we deduce that
	$$2(\alpha+2)q_\alpha|\nabla \log r|^2=2(\alpha+2)\langle\nabla p_\alpha,\nabla \log r\rangle+2(\alpha+2)^2p_\alpha|\nabla \log r|^2.$$
	Moreover we have
	$$\frac{m(\alpha+2)}{r^2}=m(\alpha+2)(|\nabla \log r|^2+p^2).$$
	Substituting the last two equations into \eqref{qa} and taking into account
	$$2\langle\nabla p_\alpha,\nabla\log r\rangle=2(q_\alpha-2p_\alpha)|\nabla\log r|^2-2\alpha p_\alpha|\nabla \log r|^2,$$
	we obtain the evolution equation for $p_\alpha$.
	
	Let us now compute the evolution equation of $h_\alpha$. From the evolution equations of $r$ and $h$ we get
	\begin{eqnarray}
	\frac{dh_\alpha}{dt}
	&=&-(\alpha+1)\frac{h}{r^{2+\alpha}}\Bigl(\Delta r-mrp^2\Bigr)\nonumber\\
	&&+\frac{1}{r^{1+\alpha}}\left\{\lap h+(m-1)\langle\nabla h,\nabla \log r\rangle\phantom{\frac{1}{r}}\right.\nonumber\\
	&&+\left.h\Bigl(-(m-1) |\nabla \log r|^2+\frac{1}{m}h^2+\frac{m-1}{m}(k-p)^2\Bigr)\right\}\nonumber\\
	&=&-(\alpha+1)\frac{h}{r^{2+\alpha}}\,\Delta r+\frac{1}{r^{1+\alpha}}\,\lap h+\frac{m-1}{r^{1+\alpha}}\langle\nabla h,\nabla \log r\rangle\label{eq h1}\\
	&&\hspace{-20pt}+h_\alpha\Bigl(-(m-1) |\nabla \log r|^2+\frac{h^2}{m}+\frac{m-1}{m}(k-p)^2+m(\alpha+1)p^2\Bigr).\nonumber
	\end{eqnarray}
	Moreover we have
	\begin{eqnarray*}
	\Delta h_\alpha&=&-(\alpha+1)\frac{h}{r^{2+\alpha}}\,\Delta r+(\alpha+1)(\alpha+2)h_\alpha|\nabla \log r|^2\nonumber\\
	&&+\frac{1}{r^{1+\alpha}}\Bigl(\Delta h-2(\alpha+1)\langle\nabla h,\nabla \log r\rangle\Bigr).
	\end{eqnarray*}
	Substituting the last equation into \eqref{eq h1} we get
	\begin{eqnarray}
	\frac{dh_\alpha}{dt}
	&=&\Delta h_\alpha+\frac{2\alpha+m+1}{r^{1+\alpha}}\langle\nabla h,\nabla\log r\rangle\nonumber\\
	&&-(\alpha^2+3\alpha+m+1) h_\alpha|\nabla \log r|^2\nonumber\\
	&&+h_\alpha\Bigl(\frac{1}{m}h^2+\frac{m-1}{m}(k-p)^2+m(\alpha+1)p^2\Bigr).\label{eq h2}
	\end{eqnarray}
	Note that
	\begin{eqnarray*}
	&&\frac{1}{m}h^2+\frac{m-1}{m}(k-p)^2+m(\alpha+1)p^2\\
	&=&\frac{1}{m}(h^2-(k-p)^2)+(k-p)^2+m(\alpha+2)p^2-mp^2\\
	&=&2p(k-p)+(k-p)^2+m(\alpha+2)p^2.
	\end{eqnarray*}
	Then the evolution equation for $h_\alpha$ follows from \eqref{eq h2},  
	$$h_\alpha\cdot 2p(k-p)=2phq_\alpha$$ 
	and from the identity
	\begin{equation}\label{eq gradh}
	\frac{1}{r^{1+\alpha}}\langle\nabla h,\nabla \log r\rangle=	\langle\nabla h_\alpha,\nabla \log r\rangle+(\alpha+1)h_\alpha |\nabla \log r|^2.
	\end{equation}
Finally, the evolution equation for $q_\alpha=h_\alpha-mp_\alpha$ follows directly from those of $p_\alpha$ and $h_\alpha$.
\end{proof}
\begin{lemma}\label{est2} Let the dimension $m$ of the equivariant Lagrangian spheres be at least two.
	\begin{enumerate}[\rm(a)]
		\item 
		Suppose $\inf_{t=0}r^{-1}p>0$. Then $\inf_{t=t_0}p_\alpha$ is attained for any $\alpha>0$ and any $t_0\in[0,T)$ and
		$$\min_{t=t_0}p_\alpha\ge \min_{t=0}p_\alpha.$$
		\item
		Suppose $\inf_{t=0}r^{-1}h>0$. Then $\inf_{t=t_0}h_\alpha$ is attained for any $\alpha>0$ and any $t_0\in[0,T)$ and
		$$\min_{t=t_0}h_\alpha\ge \min_{t=0}h_\alpha.$$
		\item
		Suppose $\inf_{t=0}r^{-1}p>0$ and $\inf_{t=0}(k-p)r^{-1}>0$. Then $\inf_{t=t_0}p_\alpha$ and $\inf_{t=t_0}q_\alpha$ are attained for any $\alpha>0$ and any $t_0\in[0,T)$ and
		$$\min_{t=t_0} q_\alpha\ge \min\left\{\min_{t=0} q_\alpha,\min_{t=0} p_\alpha/2\right\}.$$
	\end{enumerate}
\end{lemma}
\begin{proof}
	Since $m>1$, Lemma \ref{lemma profile} implies that for all $t\in[0,T)$ the profile curves are given by point-symmetric functions $z:[-\pi,\pi]\to\C$ such that $z^{-1}(0)=\{-\pi,0,\pi\}$. In particular, locally around each of the zeros of $z$ we can introduce a signed distance function $r$ and consequently from the real analyticity of $z$ we get a Taylor expansion of $p$, $k$, $h$ in terms of $r$ at these points.
	\begin{enumerate}[(a)]
		\item
		By assumption $\inf_{t=0}r^{-1}p>0$. Thus the first coefficient $c_1$ in the Taylor expansion 
		$$p(r)=\sum _{n=0}^\infty c_{2n+1}r^{2n+1}$$ 
		is positive. Since the coefficients smoothly depend on $t$ and $\alpha>0$, the function $p_\alpha$ will tend to $+\infty$ at points on the curve with $r=0$ on some maximal time interval $[0,\tau)$ where $0<\tau\le T$. Therefore, for any $t_0\in[0,\tau)$ the infimum $\inf_{t=t_0}p_\alpha$ will be attained at some point $s_0$ with $r(s_0)>0$. At such a point we get
		$$\nabla p_\alpha(s_0)=0\quad\text{and}\quad \Delta p_\alpha(s_0)\ge 0.$$
		Since \eqref{eq gradpa} gives at $s_0$
		$$0=\langle\nabla p_\alpha,\nabla \log r\rangle=\big(q_\alpha-(2+\alpha)p_\alpha\big){|\nabla \log r|^2},$$
		we conclude from Lemma \ref{lemma 4} that at $s_0$
		$$\frac {d p_\alpha}{dt}\ge p_\alpha\big\{\alpha(\alpha+m+2)|\nabla \log r|^2+m(\alpha+2)p^2+(k-p)^2\big\}.$$
		So $p_\alpha$ cannot admit a positive minimum that is decreasing in time. Consequently,
		$$\inf_{[0,\tau)}p_\alpha\ge\min_{t=0}p_\alpha>0.$$
		Since this holds for any $\alpha>0$, we may let $\alpha$ tend to zero and obtain that the first coefficient $c_1$ in the Taylor expansion of $p$ is non-decreasing in time and hence strictly positive for all $t$. In particular $\tau$ must coincide with $T$ and our estimate holds for all $t_0< T$.\\
		\item
		Similarly as above, from the assumption $\alpha>0$ and $\inf_{t=0}r^{-1}h>0$, we conclude that the function $h_\alpha$ will tend to $+\infty$ at points on the curve with $r=0$ on some maximal time interval $[0,\tau)$ where $0<\tau\le T$. In addition for any $t_0\in[0,\tau)$ the infimum $\inf_{t=t_0}h_\alpha$ will be attained at some point $s_0$ with $r(s_0)>0$. Observe that
		$$2phq_\alpha=2p(k-p)h_\alpha$$
		and
		$$m(\alpha+2)p^2+(k-p)^2+2p(k-p)=k^2+\bigl(m(\alpha+2)-1\bigr)p^2.$$
		Then from Lemma \ref{lemma 4} we deduce that at a minimum $s_0$ of $h_\alpha$ we have
		$$\frac {dh_\alpha}{dt}\ge
		h_\alpha\big\{\alpha(\alpha+m)|\nabla \log r|^2+k^2+\bigl(m(\alpha+2)-1\bigr)p^2\big\}.$$
		Therefore $h_\alpha$ cannot attain a decreasing positive minimum at points where $r> 0$. The rest of the proof is the same as in part (a).\\
		\item
		First note that 
		$$m(\alpha+2)p^2+(k-p)^2+2ph=\bigl(m(\alpha+4)-1\bigr)p^2+k^2\ge 0.$$
		Therefore, at a point of a positive minimum of $q_\alpha$ we conclude from part (a) that
		$$\big(\alpha(\alpha+m)-2m\big)\min_{t=t_0}q_\alpha+4mp_\alpha\ge 2m\big(2\min_{t=0} p_\alpha-\min_{t=t_0} q_\alpha\big).$$
		Hence, if $\min_{t=t_0} q_\alpha$ is positive but smaller than $\min_{t=0} p_\alpha$, then the minimum cannot be decreasing. Consequently, we obtain the desired lower bound for the minimum of $q_\alpha$. 
	\end{enumerate}
This completes the proof.
\end{proof}

For $\alpha= 0$, Lemma \ref{est1} and Lemma \ref{est2} imply the following estimate on the Ricci curvature.

\begin{lemma}\label{lemm ricci}
	Let $F_0:\Sf^m\to\C^m$, $m>1$, be an equivariant Lagrangian immersion such that the Ricci curvature satisfies $\operatorname{Ric}\ge cr^2\cdot\gind,$ where $c$ is a positive constant. Then for all $t_0\in[0,T)$ we have
	$$\inf_{t=t_0}\frac{p}{r}\ge\inf_{t=0}\frac{p}{r}>0,\quad\inf_{t=t_0}\frac{k-p}{r}\ge\inf_{t=0}\frac{k-p}{r}>0.$$
	In particular the Ricci curvature of the Lagrangian submanifolds evolving under the mean curvature flow satisfies $\operatorname{Ric}\ge \varepsilon r^2\cdot\gind$, with a positive constant $\varepsilon$ not depending on $t$.
\end{lemma}

\subsection{Angles and areas}
	Recall that the profile curve $z:[-\pi,\pi]\to\C^m$ of a Lagrangian sphere is point symmetric with respect to the origin. So it can be divided into the union of the two arcs $\gamma_\ell=z\big([0,\pi]\big)$ and $-\gamma_\ell=z\big([-\pi,0]\big)$. It is clear that it is sufficient to study one of these arcs under the equivariant curve shortening flow.

	Let $\alpha_0$ denote the angle between the unit normal $\nu$ of $\gamma_\ell$ at $s=0$ and the unit vector $e_1=\partial/\partial x$.  Without loss of generality we may assume that $\alpha_0\in(0,\pi/2)$, so that as in Remark \ref{rem butterfly}(b) for sufficiently small $s$ we have $x(s)<0$ and $y(s)>0$.  Then for any $s\in[0,\pi]$ we define the {\it Umlaufwinkel}
	$$\alpha(s):=\alpha_0+\int\limits_0^sk(s)d\mu(s),$$
	Then for all $s\in[0,\pi]$ the function $\alpha$ measures the angle between the unit normal $\nu$ and $e_1$ modulo some integer multiple of $2\pi$.
\begin{figure}[ht]
	\includegraphics[scale=.5]{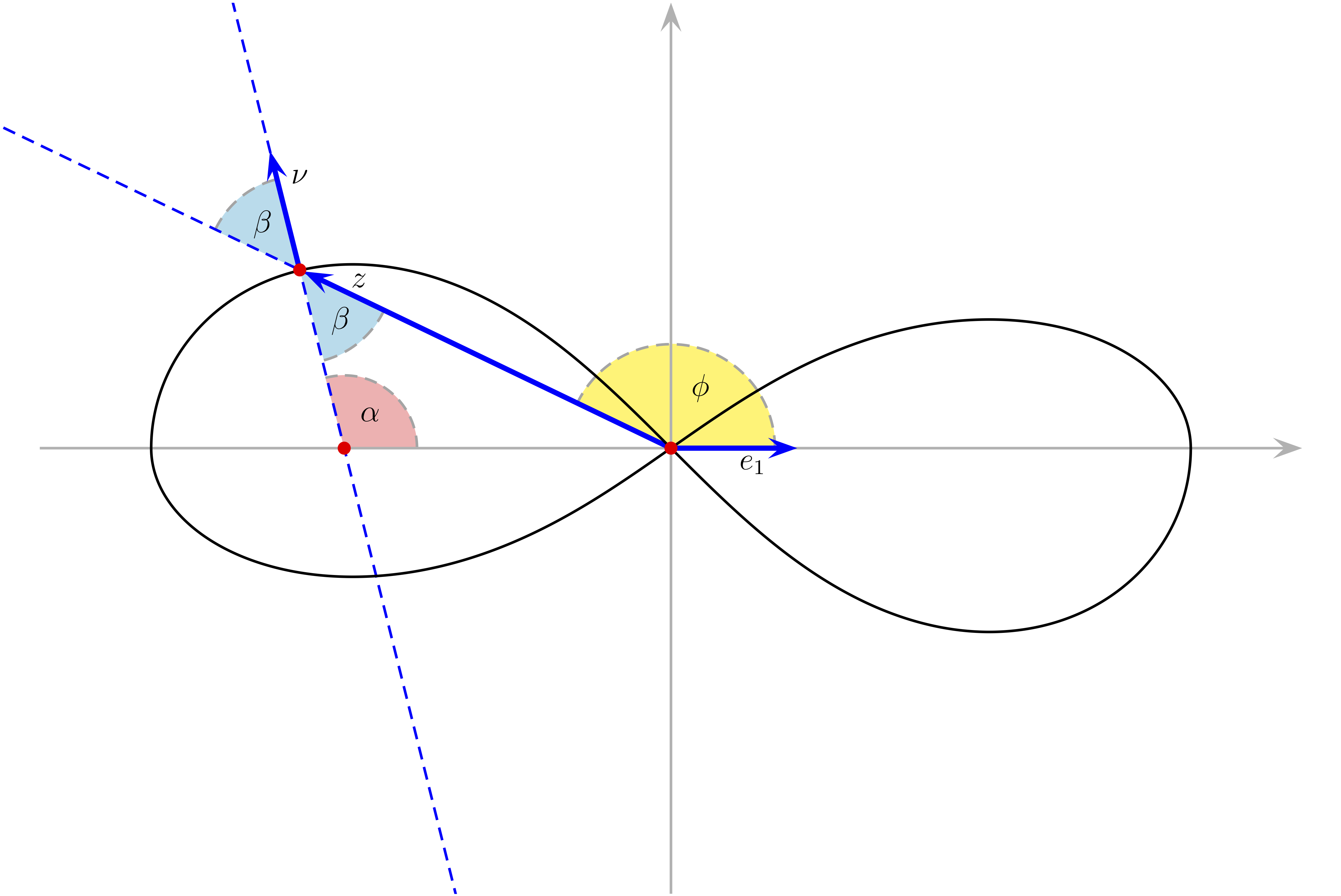}\caption{The relation between the angles $\alpha$, $\beta$ and $\phi$ is $\phi=\alpha+\beta$.}\label{fig angles}
\end{figure}
Similarly we define the {\it polar angle} by
\begin{equation}
\phi(s):=\phi_0+\int\limits_0^sp(s)d\mu(s).\label{eq phi}
\end{equation}
where
$$\phi_0:=\alpha_0+\pi/2.$$
Then for any $s\in[0,\pi]$ the quantity $\phi(s)$ is the angle between $e_1$ and the normalized position vector $z(s)/r(s)$, again up to some
integer multiple of $2\pi$. For $s\to 0$ the normalized position vector converges to the unit tangent vector $e$ at $s=0$.

In addition we define 
$$\beta(s):=\phi(s)-\alpha(s).$$
Then up to multiples of $2\pi$ the angle $\beta$ is the angle between the unit normal $\nu(s)$ and the normalized position vector $z(s)/r(s)$, i.e. 
$\cos\beta=rp$
(compare with Figure \ref{fig angles}). In particular $\beta_0:=\beta(0)=\pi/2$.

The condition $k>0$ implies that the Umlaufwinkel $\alpha$ is a strictly increasing function in $s$. The polar angle $\phi$ is strictly increasing, if $p>0$. In view of 
$$\beta=\phi-\alpha,\quad d\alpha=kd\mu\quad\text {and} \quad d\phi=pd\mu$$ 
we conclude that the condition $k-p>0$ implies that the angle $\beta$ is strictly decreasing in $s$. From this, Lemma \ref{est1} and Lemma \ref{est2} we conclude the following result.
\begin{lemma}\label{lemm opening}
	Let $F_0:\Sf^m\to\C^m$, $m>1$, be an equivariant Lagrangian immersion such that the Ricci curvature satisfies $\operatorname{Ric}\ge cr^2\cdot\gind,$ where $c$ is a positive constant. Then, under the equivariant curve shortening flow, the angles $\alpha$ and $\phi$ remain strictly increasing functions on $[0,\pi]$ for all $t\in[0,T)$, whereas the angle $\beta$ stays strictly decreasing on $[0,\pi]$. 
\end{lemma}
\begin{definition}
	Let $z:[-\pi,\pi]\to\C$ be a profile curve of an equivariant Lagrangian sphere $F_0:\Sf^m\to\C^m$, $m>1$. We define the {\sl opening angle} $\phi_\ell$ of the loop $\gamma_\ell:=z([0,\pi])$ to be
	$$\phi_\ell:=\phi(\pi)-\phi(0)=\int\limits_0^\pi p(s)d\mu(s)=\int\limits_{\gamma_\ell}pd\mu.$$
\end{definition}
Since the opening angle evolves in time we want to understand its behavior in more detail. To achieve this we will apply the divergence theorem several times. Note that the divergence theorem implies for any smooth function $f$ the identity
\begin{equation}\label{eq div}
\int\limits_{\gamma_\ell}\operatorname{div}\left(\nabla f-\frac{f}{r}\nabla r\right)d\mu=0.
\end{equation}

According to Lemma \ref{lemma 4} $d\mu$ is evolving by
$$\dt d\mu=-hk d\mu.$$
Then a straightforward computation shows that for any constant $\varepsilon$ the $1$-form $r^\varepsilon p d\mu$ satisfies
\begin{equation}\nonumber
	\dt (r^\varepsilon p d\mu)=\Bigl(\operatorname{div}(r^{\varepsilon-1}h\nabla r)-\varepsilon r^{\varepsilon-2}h\Bigr) d\mu.
\end{equation}
For $\varepsilon=2$ and $\varepsilon=0$ we derive the evolution equations
\begin{eqnarray*}
	\dt(\langle z,\nu\rangle d\mu)
	&=&\Bigl(\operatorname{div}(rh\nabla r)-2h\Bigr) d\mu,\nonumber\\
	\dt(pd\mu)
	&=&\Bigl(\operatorname{div}(h\nabla \log r)\Bigr) d\mu.\nonumber\\
\end{eqnarray*}
From the last equation we can now proceed to compute the evolution of the opening angle $\phi_\ell$. We get
\begin{eqnarray}
\dt\phi_\ell&=&\dt\int\limits_{\gamma_\ell}pd\mu=\int\limits_{\gamma_\ell}\operatorname{div}(h\nabla \log r)d\mu\label{eq evolphi}\\
&=&\int\limits_{\gamma_\ell}\operatorname{div}\left(\frac{k-3p}{r}\nabla r\right)d\mu+(m+2)\int\limits_{\gamma_\ell}\operatorname{div}\left(\frac{p}{r}\nabla r\right)d\mu,\nonumber\\
&=&-(m+2)\left(\lim_{s\searrow 0}\frac{p(s)}{r(s)}+\lim_{s\nearrow \pi}\frac{p(s)}{r(s)}\right),\nonumber
\end{eqnarray}
where we have used Lemma \ref{lemm crucial}(b),
$$\lim_{r\to 0}|\nabla r|^2=1$$
and the divergence theorem in the last step. Now let $c_0(t)$ denote the first order coefficient in the Taylor expansion of the function $p$ in terms of $r$ at the point $s=0$, i.e. 
$$p(r,t)=c_0(t)r+\sum_{n=1}^\infty a_{2n+1}(t)r^{2n+1}$$ 
for some smooth time dependent functions $a_{2n+1}(t)$. Likewise let $c_\pi(t)$ denote the first order coefficient in the Taylor expansion of $p$ in terms of $r$ at the point $s=\pi$, i.e. $$p(r,t)=c_\pi(t)r+\sum_{n=1}^\infty b_{2n+1}(t)r^{2n+1}$$
for some smooth time dependent functions $b_{2n+1}(t)$.

A direct consequence of this and Lemma \ref{lemm ricc} is the following conclusion
\begin{lemma}
	The opening angle $\phi_\ell$ of the loop $\gamma_\ell$ evolves by
	$$\dt\phi_\ell=-(m+2)\bigl(c_0(t)+c_\pi(t)\bigr),$$
	where $c_0(t)$, respectively $c_\pi(t)$ denote the first order coefficients in the Taylor expansion of $p$ around the points $s=0$ respectively $s=\pi$. In particular, if the initial Ricci curvature satisfies $\operatorname{Ric}\ge cr^2\cdot\gind$ for some positive constant $c$, then in dimension $m>1$ we obtain
	$$\dt\phi_\ell\le-2(m+2)\inf_{t=0}\frac{p}{r}<0.$$
\end{lemma}
From the last lemma we see that the initial condition $\operatorname{Ric}\ge cr^2\cdot\gind$ on the Ricci curvature implies that the opening angle is decreasing in time. This gives another geometric interpretation of our assumption in the main theorem.

As above, suppose we have parametrized $z$ in such a way that $\nu$ is the outward unit normal along $\gamma_\ell$. Let $A_\ell$ denote the area (with sign) enclosed by the loop $\gamma_\ell$. By the divergence theorem we must have
$$A_\ell=\frac{1}{2}\int\limits_{\gamma_\ell}\langle z,\nu\rangle d\mu.$$
Therefore
$$\frac{d A_\ell}{dt}=\frac{1}{2}\int\limits_{\gamma_\ell}\operatorname{div}(rh\nabla r)d\mu
-\int\limits_{\gamma_\ell}h d\mu.$$
Applying the divergence theorem once more and taking into account that the quantity $rh\nabla r$ tends to zero as $r$ tends to zero, we get
\begin{equation}\label{eq evolarea}
\frac{d A_\ell}{dt}=-\int\limits_{\gamma_\ell}h d\mu.
\end{equation}
On the other hand we have
\begin{equation}\label{eq h}
\int\limits_{\gamma_\ell}hd\mu=\int\limits_{\gamma_\ell}(k-p)d\mu+m\int\limits_{\gamma_\ell}pd\mu=\int\limits_{\gamma_\ell}(k-p)d\mu+m\phi_\ell.
\end{equation}
The evolution equation for $k$ and $d\mu$ in Lemma \ref{lemma 4} imply
$$\dt\int\limits_{\gamma_\ell}kd\mu=\int\limits_{\gamma_\ell}\Delta h\, d\mu\overset{\eqref{eq div}}{=}\int\limits_{\gamma_\ell}\operatorname{div}(h\nabla\log r)\, d\mu\overset{\eqref{eq evolphi}}{=}\dt\int\limits_{\gamma_\ell}pd\mu.$$
Hence
$$\dt\int\limits_{\gamma_\ell}(k-p)d\mu=0$$
and from \eqref{eq evolarea}, \eqref{eq h} we obtain
$$\frac{d^2}{dt^2}A_\ell=-m\dt\phi_\ell=-m(m+2)\Bigl(c_0(t)+c_\pi(t)\Bigr).$$
In particular the initial condition $\operatorname{Ric}\ge cr^2\cdot\gind$ implies that in dimension $m>1$ the enclosed area $A_\ell$ of the loop $\gamma_\ell$ is a strictly decreasing and strictly convex function in $t$.

\section{Rescaling the singularity}
In this section we will rescale the singularities and prove Theorem A. To this end let us first recall some general facts about singularities. 

The maximal time $T$ of existence of a smooth solution to the mean curvature flow of a compact submanifold $M^m\subset\real{n}$ must be finite. The following general theorem is well known and shows how one can analyze forming singularities of the mean curvature flow by parabolic rescalings around points where the norm of the second fundamental form attains its maximum; for details see \cite{chen-he}, \cite[Section 2.16]{tao} and the references therein.

\begin{proposition}\label{blow}
	Let $F:M\times [0,T)\to \R^n$ be a solution of the mean curvature flow, where $M$ is compact, connected and $0<T<\infty$ is the maximal time of existence of a smooth solution. There exists a point $x_\infty\in M$ and a sequence $\{(x_j,t_j)\}_{j\in\natural{}}$ of points in $M\times[0,T)$ with $\lim x_j=x_\infty$, $\lim t_j=T$ such that 
	$$|A(x_j,t_j)|=\max_{M\times[0, t_j]}|A(x,t)|=:a_j\to\infty.$$
	
	Consider the family of maps $F_j:M\times[L_j,R_j)
	\to\R^n$, $j\in\natural{}$,
	given by
	$$F_j(x,\tau):=F_{j,\tau}(x):=a_j\big(F(x,\tau/{a^{2}_j}+t_j)-\Gamma_j\big),$$
	where $L_j:=-a^2_jt_j$, $R_j:=a^2_j(T-t_j)$ and $\Gamma_j:=F(x_j,t_j)$. Then the following
	holds:
	\begin{enumerate}[\rm (a)]
		\item
		The family of maps $\{F_j\}_{j\in\natural{}}$ evolve in time by the mean curvature flow. The norm $|A_j|$ of the second fundamental form of $F_j$ satisfies the equation
		$$|A_j(x,\tau)|=a^{-1}_j|A(x,\tau/a_j^2+t_j)|.$$
		Moreover, for any $\tau\le 0$ we have $|A_j(x,\tau)|\le 1$ and $|A_j(x_j,0)|=1$ for any $j\in\natural{}$.
		\medskip
		\item
		For any fixed $\tau\le 0$, the sequence of pointed Riemannian manifolds
		$$\big\{M,F^*_{j,\tau}\langle\,\cdot,\cdot\rangle,\Gamma_j\big\}_{j\in\natural{}}$$
		smoothly subconverges in the Cheeger-Gromov sense to a connected complete pointed Riemannian manifold $(M_\infty,g_\infty,\Gamma_\infty)$ that does not depend on the
		choice of $\tau$.
		\medskip
		\item
		There is an ancient solution $F_\infty:M_\infty\times(-\infty,0]\to\R^n$ of the mean curvature flow, such that for each fixed time $\tau\le 0$, the sequence $\{F_{j,\tau}\}_{j\in\natural{}}$ smoothly subconverges in the Cheeger-Gromov sense to $F_{\infty,\tau}$. 
		Additionally, $|A_{F_\infty}|\le 1$ and $|A_{F_\infty}(x_\infty,0)|=1.$
		\medskip
		\item
		 If the singularity is of type-II, then $R_j\to\infty$ and the limiting flow $F_\infty$ can be constructed on the whole time axis $(-\infty,\infty)$ and thus gives an eternal solution of the mean curvature flow. 
	\end{enumerate}
\end{proposition}
The following lemma will be crucial for the classification of the blow-up curves.
\begin{lemma}\label{lemm ray}
	Let $z:[0,\infty)\to\C$ be a real analytic curve, parametrized by arc length $\sigma$ such that $z(0)=0$, $k\ge p\ge 0$ and
	$$\int\limits_0^\infty p(\sigma)d\sigma<\infty.$$
	Then $p\equiv 0$ and $z$ is a ray.
\end{lemma}
\begin{proof}
	As in \eqref{eq phi} let
	$$\phi(\sigma)=\phi_0+\int\limits_0^\sigma p(\sigma)d\sigma$$
be the polar angle and define
	$$\phi_\infty:=\phi_0+\int\limits_0^\infty p(\sigma)d\sigma.$$
	Since $p\ge 0$, $\phi$ is an increasing function.
	
	We will show $\phi\equiv\phi_0$, hence $p\equiv 0$, and distinguish two cases.
	
	{\bf Case 1.} Suppose there exists $\sigma_0\in[0,\infty)$ such that $\phi(\sigma_0)=\phi_\infty$. Then from the monotonicity of $\phi$ we get $\phi(\sigma)=\phi_\infty$ for all $\sigma\ge \sigma_0$ and then the real analyticity implies $\phi(\sigma)=\phi_\infty$ for all $\sigma\in[0,\infty)$. But then in particular $\phi_0=\phi_\infty$ and $p\equiv 0$.
	
	{\bf Case 2.} In this case we have $\phi(\sigma)<\phi_\infty$ for all $\sigma\in[0,\infty)$. Since $\phi$ is monotone and bounded, there exists $\sigma_0\in[0,\infty)$ such that $z_{|[\sigma_0,\infty)}$ can be represented as a graph over the flat line passing through the origin in direction of $v=(\cos\phi_\infty,\sin\phi_ \infty)$. After a rotation around the origin we may assume without loss of generality that $\phi_\infty=\pi$ such that $z_{|[\sigma_0,\infty)}$ can be represented as the graph of a real analytic function $y:(-\infty,s_0]\to\real{}$, $s_0<0$, where $y$ is a strictly decreasing concave function, because $k\ge p\ge 0$ (see Figure \ref{fig limit}).
	
	On the interval $(-\infty,s_0]$ the function $rp=\cos\beta$ is given by
	$$\cos\beta=\frac{sy'-y}{\sqrt{(1+(y')^2)(s^2+y^2)}}.$$
	Since $y$ is concave, $s<0$ and $(sy'-y)'=sy''\ge 0$, we conclude that for $s<s_0$ we must have
	\begin{equation}\label{eq estimate 1}
	y'(s)\ge \frac{y(s)}{s} +\frac{s_0y'(s_0)-y(s_0)}{s}.
	\end{equation}
	Since $y'\le 0$ and $y''\le 0$ we observe that
	$$a:=\lim_{s\to-\infty}y'(s)$$
	exists and is non-positive. From inequality \eqref{eq estimate 1} we conclude that
	$$a\ge\lim_{s\to-\infty}\frac{y(s)}{s}.$$
	Now for $\tau<s_0$ let
	$$\epsilon(s):=y'(\tau)s+b(\tau)$$
	be the tangent line to the graph of the function $y$ passing through the point $z(\tau)=(\tau,y(\tau))$ (see Figure \ref{fig limit}). 
	\begin{figure}[ht]
		\includegraphics[scale=.5]{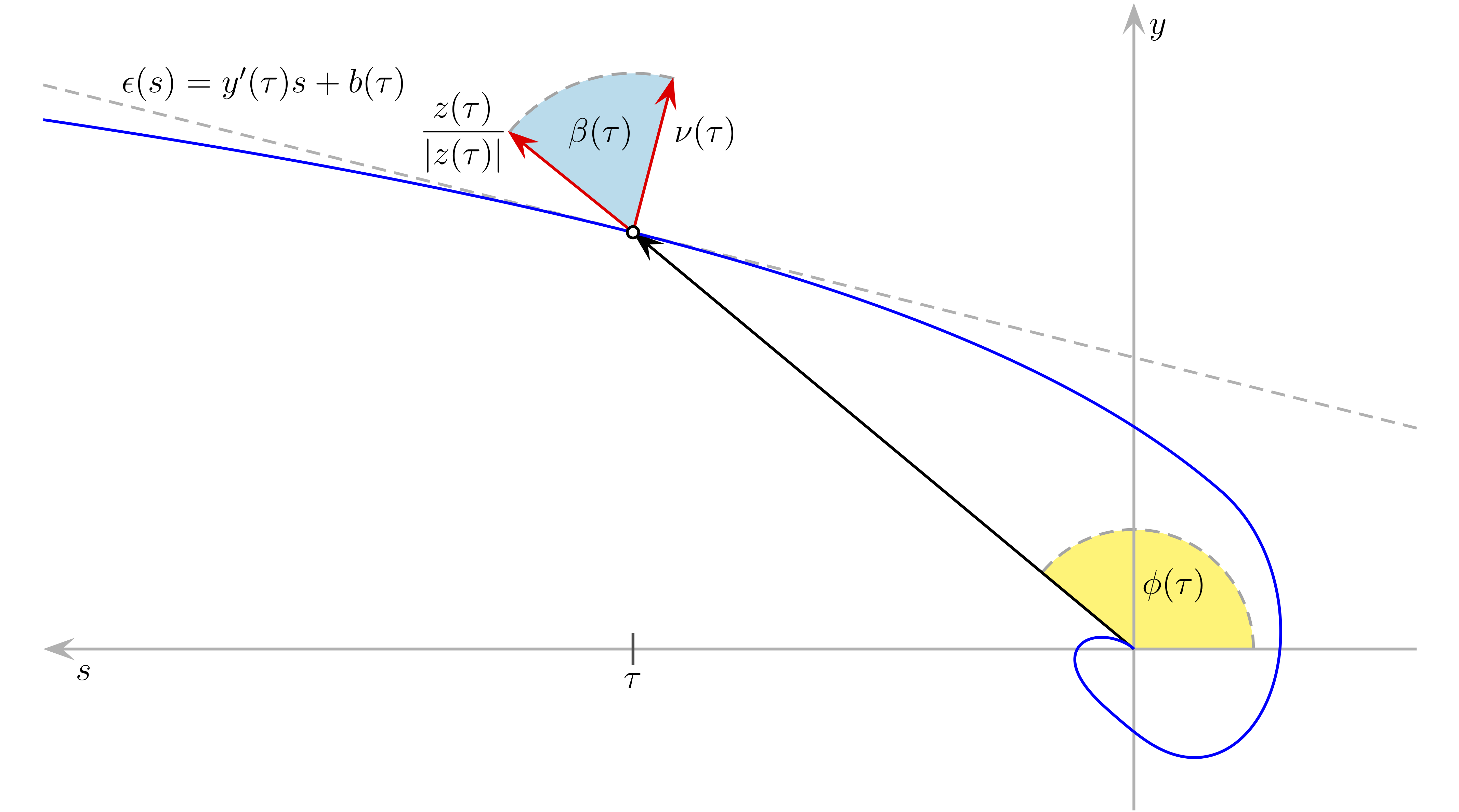}\caption{The angle $\beta$ between the normal vector and the position vector tends to $\pi/2$ for $s\to-\infty$.}\label{fig limit}
	\end{figure}
	Since $y$ is concave we derive $y(s)\le \epsilon(s),$
	for all $s\le s_0.$ Therefore we obtain the estimate
	\begin{equation}\label{eq estimate 2}
	\frac{y(s)}{s}\ge y'(\tau)+\frac{b(\tau)}{s},\quad\text{ for all }s\le \tau.
	\end{equation}
	Passing to the limit we obtain that
	$$\lim_{s\to-\infty}\frac{y(s)}{s}\ge y'(\tau).$$
	Since the above inequality holds for any $\tau<s_0$ we may let $\tau$ tend to $-\infty$ which implies
	$$\lim_{s\to-\infty}\frac{y(s)}{s}\ge a.$$
	Thus we have shown
	$$\lim_{s\to-\infty}(y'(s)-y(s)/s)=0$$
	and this yields
	$$\lim_{s\to-\infty}\cos\beta(s)=0.$$
	We will now see that this implies that $\cos\beta\equiv 0$ on the whole curve $z$. First note that the graphical property implies that for sufficiently large $\sigma$ the distance function $r$ is a strictly increasing function in terms of the arc length parameter $\sigma$, so that $r'=\partial r/\partial\sigma>0$. Since $(\cos\beta)'=(rp)'=(k-p)r'$ we conclude that the function $\cos\beta$ is increasing for sufficiently large $\sigma$. Because $\cos\beta=rp\ge 0$ and $\lim_{\sigma\to\infty}\cos\beta(\sigma)=0$, this implies $\cos\beta(\sigma)=0$ for all sufficiently large $\sigma$ and then the real analyticity also implies the same for all $\sigma\in[0,\infty)$. Hence $p\equiv 0$ and $\phi\equiv\phi_0$ as claimed.
\end{proof}

{\bf Proof of Theorem A.}
	Suppose that $F:\Sp^m\times[0,T)\to\C^m$ is the evolution by Lagrangian mean curvature flow of a Lagrangian sphere satisfying the assumptions of Theorem A. Let
	$$a_j=\max_{\Sp^m\times[0, t_j]}|A|,$$
	where $\{t_j\}_{j\in\natural{}}$ is an increasing sequence with $\lim_{j\to\infty}t_j=T$. Following Proposition \ref{blow}, take a sequence $\{x_j\}_{j\in\natural{}}$ in $\Sf^m$ such that
	$$|A|(x_j,t_j)=a_j.$$
	Let
	$$F_j(x,\tau)=a_j\bigl(F(x,\tau/a_j^2+t_j)-F(x_j,t_j)\bigr)$$
	be the rescaled Lagrangian spheres and denote the distance between the rescaled blow-up point (the origin) and the rescaled intersection point $-a_jF(x_j,t_j)$ by
	$$r_j:=a_jr(x_j,t_j).$$
	{\bf Claim.} $\limsup_{j\to\infty}r_j=\infty$.

	Indeed, suppose to the contrary that
	$\limsup_{j\to\infty} r_j=r_\infty<\infty.$
	Then, following Proposition \ref{blow}, we obtain an eternal {\it equivariant} Lagrangian mean curvature flow $F_\infty$ generated by limit profile curves $z_\infty:M_\infty\to\C$, where $M_\infty$ is a complete pointed Riemannian manifold
	of dimension one. Since $F_\infty$ is an eternal solution and eternal solutions in euclidean space are known to be non-compact, $M_\infty$ is diffeomorphic to $\real{}$. The limit profile curves $z_\infty$ evolve again by an equivariant curve shortening flow with a different center, which without loss of generality can be assumed to be the origin. Because $k$ and $p$ have the same scaling behavior, the limit flow still satisfies
	$$rk\ge rp\ge 0.$$
	For any fixed time $\tau\in\real{}$ there exists a local parametrization of $z_\infty(M_\infty\times\{\tau\})$ by a real analytic arc $z:[0,\infty)\to\C$, parametrized by arc length $\sigma$ such that $z(0)=0$ and $k\ge p\ge 0$ and $k\not\equiv 0$, because $|A|^2=k^2+3(m-1)p^2$. Moreover, since by Lemma \ref{lemm opening} the opening angle is decreasing, we conclude that 
	$$\int\limits_0^\infty p(\sigma)d\sigma<\infty.$$
	Applying Lemma \ref{lemm ray} and using the point-symmetry we obtain that $z_\infty(M_\infty\times\{\tau\})$ must be straight line, which is a contradiction. This proves the claim.
	
	Instead of taking a blow-up sequence for $F$ we may, equivalently, consider a blow-up sequence for the corresponding profile curves $z$. Therefore consider a sequence of points $\{s_j\}_{j\in\natural{}}$ in $\Sf^1$ such that
	$$\bigl(k^2+(m-1)p^2\bigr)(s_j,t_j)=a_j^2.$$
	Following the statements of Proposition \ref{blow}, let us define the family of curves $z_j:\Sf^1\times[L_j,R_j)\to\C$, $j\in\natural{}$, given by
	$$z_j(s,\tau):=z_{j,\tau}(s):=a_j\big(z(s,\tau/{a^{2}_j}+t_j)-\gamma_j\big),$$
	where $L_j=-a^2_jt_j$, $R_j=a^2_j(T-t_j)$ and $\gamma_j=z(s_j,t_j)$. One can easily verify that the curves $\{z_j\}_{j\in\natural{}}$ evolve by the equation
	\begin{equation*}
	\frac{dz_j}{d\tau}=-\left(k_j+(m-1)\frac{\langle z_j+a_j\gamma_j,\nu_j\rangle}{|z_j+a_j\gamma_j|^2}\right)\nu_j,
	\end{equation*}
	where
	$$k_j(s,\tau)=a^{-1}_jk(s,\tau/a_j^2+t_j)$$
	is the signed curvature and
	$$\nu_j(s,\tau)=\nu(s,\tau/a_j^2+t_j)$$
	is the outer unit normal of the rescaled curve $z_{j,\tau}$ at $s$. Note that with the notation from above we have 
	$$r_j=a_j|\gamma_j|=a_jr(s_j,t_j).$$
	From Cauchy-Schwarz' inequality we obtain 
	$$\left|\frac{\langle z_j+a_j\gamma_j,\nu_j\rangle}{|z_j+a_j\gamma_j|^2}\right|\le\frac{1}{|z_j+a_j\gamma_j|}.$$
	For large $r_j$ we can thus estimate
	$$\left|\frac{\langle z_j+a_j\gamma_j,\nu_j\rangle}{|z_j+a_j\gamma_j|^2}\right|\le\frac{1}{r_j-|z_j|}.$$
	So for each $s\in\Sp^1$ where $z_j(s,\tau)$ converges as $j\to\infty$ we conclude that 
	$$\lim_{j\to\infty}\frac{\langle z_j+a_j\gamma_j,\nu_j\rangle}{|z_j+a_j\gamma_j|^2}=0.$$
	Since the intersection point after rescaling tends to infinity, this and the point-symmetry show that the limit flow $z_\infty$ in this case is an eternal and weakly convex non-flat solution of the curve shortening flow. It is well known by results of Altschuler \cite{altschuler} and Hamilton \cite{hamilton} that such a solution is a grim reaper. It is clear that the parabolic rescaling of the Lagrangian spheres then converge in the Cheeger-Gromov sense to the product of a grim reaper with a flat Lagrangian space.
\hfill$\square$\newline
\bibliographystyle{alpha}

\begin{bibdiv}
\begin{biblist}

\bib{altschuler}{article}{
	author={Altschuler, S. J.},
	title={Singularities of the curve shrinking flow for space curves},
	journal={J. Differential Geom.},
	volume={34},
	date={1991},
	number={2},
	pages={491--514},
}

\bib{anciaux}{article}{
   author={Anciaux, H.},
   title={Construction of Lagrangian self-similar solutions to the mean
   curvature flow in $\mathbb{C}^n$},
   journal={Geom. Dedicata},
   volume={120},
   date={2006},
   pages={37--48},
}

\bib{angenent}{article}{
	author={Angenent, S.},
	title={On the formation of singularities in the curve shortening flow},
	journal={J. Differential Geom.},
	volume={33},
	date={1991},
	number={3},
	pages={601--633},
}

\bib{bcm}{article}{
	author={Borrelli, V.},
	author={Chen, B.-Y.},
	author={Morvan, J.-M.},
	title={Une caract\'erisation g\'eom\'etrique de la sph\`ere de Whitney},
	language={French, with English and French summaries},
	journal={C. R. Acad. Sci. Paris S\'er. I Math.},
	volume={321},
	date={1995},
	number={11},
	pages={1485--1490},
}


\bib{chen}{article}{
	author={Chen, B.-Y.},
	title={Complex extensors and Lagrangian submanifolds in complex Euclidean
		spaces},
	journal={Tohoku Math. J. (2)},
	volume={49},
	date={1997},
	number={2},
	pages={277--297},
}

\bib{chen-he}{article}{
   author={Chen, J.},
   author={He, W.},
   title={A note on singular time of mean curvature flow},
   journal={Math. Z.},
   volume={266},
   date={2010},
   number={4},
   pages={921--931},
}

\bib{evans}{article}{
         author={Evans, C.}
	author={Lotay, J.},
	author={Schulze, F.},
	title={Remarks on the self-shrinking Clifford torus},
	journal={arXiv:1802.01423},
	date={2018},
	pages={1--26},
}

\bib{groh}{book}{
    author = {Groh, K.},
    title = {Singular behavior of equivariant Lagrangian mean curvature flow},
    pages = {85},
    date = {2007},
    Publisher = {Hannover: Univ. Hannover, Fakult\"at f\"ur Mathematik und Physik (Diss.)},
}
	
\bib{gssz}{article}{
   author={Groh, K.},
   author={Schwarz, M.},
   author={Smoczyk, K.},
   author={Zehmisch, K.},
   title={Mean curvature flow of monotone Lagrangian submanifolds},
   journal={Math. Z.},
   volume={257},
   date={2007},
   number={2},
   pages={295--327},
}

\bib{hamilton}{article}{
	author={Hamilton, R.},
	title={Harnack estimate for the mean curvature flow},
	journal={J. Differential Geom.},
	volume={41},
	date={1995},
	pages={215--226},
}

\bib{joyce}{article}{
	author={Joyce, D.},
	author={Lee, Y.-I.},
	author={Tsui, M.-P.},
	title={Self-similar solutions and translating solitons for Lagrangian mean curvature flow},
	journal={J. Differential Geom.},
	volume={84},
	date={2010},
	number={1}
	pages={127--161},
}

\bib{martin}{article}{
	author={Martin, F.},
	author={Savas-Halilaj, A.},
	author={Smoczyk, K.},
	title={On the topology of translating solitons of the mean curvature flow},
	journal={Calc. Var. Partial Differential Equations},
	volume={54},
	date={2015},
	number={3}
	pages={2853--2882},
}

\bib{tian}{article}{
	author={Neves, A.},
	author={Tian, G.},
	title={Translating solutions to Lagrangian mean curvature flow},
	journal={Trans. Amer. Math. Soc.},
	volume={365},
	date={2013},
	number={11}
	pages={5655--5680},
}

\bib{neves1}{article}{
   author={Neves, A.},
   title={Singularities of Lagrangian mean curvature flow: zero-Maslov class
   case},
   journal={Invent. Math.},
   volume={168},
   date={2007},
   number={3},
   pages={449--484},
}

\bib{ru}{article}{
	author={Ros, A.},
	author={Urbano, F.},
	title={Lagrangian submanifolds of ${\bf C}^n$ with conformal Maslov form
		and the Whitney sphere},
	journal={J. Math. Soc. Japan},
	volume={50},
	date={1998},
	number={1},
	pages={203--226},
}

\bib{smoczyk3}{article}{
   author={Smoczyk, K.},
   title={Local non-collapsing of volume for the Lagrangian mean curvature flow},
   journal={ arXiv:1801.07303},
   date={2018},
   pages={1--18},
}

\bib{smoczyk1}{book}{
   author={Smoczyk, K.},
   title={Der Lagrangesche mittlere Kr\"ummungsfluss},
   publisher={Leipzig: Univ. Leipzig (Habil.-Schr.)},
   date={2000},
   pages={102},
}

\bib{smoczyk2}{article}{
   author={Smoczyk, K.},
   title={Symmetric hypersurfaces in Riemannian manifolds contracting to
   Lie-groups by their mean curvature},
   journal={Calc. Var. Partial Differential Equations},
   volume={4},
   date={1996},
   number={2},
   pages={155--170},
}

\bib{tao}{book}{
	author={Tao, T.},
	title={Poincar\'e's legacies, pages from year two of a mathematical blog.
		Part II},
	publisher={American Mathematical Society, Providence, RI},
	date={2009},
	pages={x+292},
	isbn={978-0-8218-4885-2},
}

\bib{viana}{article}{
	author={Viana, C.},
	title={A note on the evolution of the Whitney sphere along mean curvature flow},
	journal={arXiv:1802.02108},
	date={2018},
	pages={1-13},
}

\bib{xin}{article}{
   author={Xin, Y.-L.},
   title={Translating solitons of the mean curvature flow},
   journal={Calc. Var. Partial Differential Equations},
   volume={54},
   date={2015},
   pages={1995--2016},
}

\end{biblist}
\end{bibdiv}

\end{document}